\documentclass{kybernetika} 

\usepackage{comment}
\usepackage[utf8]{inputenc}
\usepackage{mathtools}
\usepackage{physics}
\usepackage{mathrsfs}

\newtheorem{theorem}{Theorem}[section]
\newtheorem{lemma}[theorem]{Lemma}

\newtheorem{corollary}[theorem]{Corollary}

\newtheorem{definition}[theorem]{Definition}

\newtheorem{assumption}[theorem]{Assumption}
\newtheorem{notation}[theorem]{Notation}

\def\digr{\mathcal{D}}
\def\trellis{\mathcal{T}}
\def\nodes{N}
\def\edges{E}
\def\trellisnodes{\mathcal{N}}
\def\trellisedges{\mathcal{E}}
\def\Rmax{\mathbb{R}_{\max}}

\begin{document}
	\pagestyle{myheadings}
	
	\title{A bound for the rank-one transient of \newline	inhomogeneous matrix products \newline in special case}
	
	\author{Arthur Kennedy-Cochran-Patrick, Serge\u{\i} Sergeev, and \v{S}tefan Bere\v{z}n\'y}
	
	\contact{Arthur}{Kennedy-Cochran-Patrick}{University of Birmingham, School of Mathematics,  Watson Building, Edgbaston, B15 2TT Birmingham. United Kingdom.}{AXC381@bham.ac.uk}
	\contact{Serge\u{\i}}{Sergeev}{University of Birmingham, School of Mathematics,  Watson Building, Edgbaston, B15 2TT Birmingham. United Kingdom.}{S.Sergeev@bham.ac.uk}
	\contact{\v{S}tefan}{Bere\v{z}n\'y}{DMTI, FEEI TUKE, N\v{e}mcovej 32, 042\,00 Ko\v{s}ice, Slovak Republic\,  and\, University of Birmingham, School of Mathematics, Watson Building, Edgbaston, B15 2TT Birmingham. United Kingdom.}{Stefan.Berezny@tuke.sk}
	
	\markboth{A. Kennedy-Cochran-Patrick, S. Sergeev, and \v{S}. Bere\v{z}n\'y}{A Bound for the rank-one transient of inhomogeneous matrix products in special case}
	
	\maketitle
	
	\begin{abstract}
		We consider inhomogeneous matrix products over max-plus algebra, where the matrices in the product satisfy certain assumptions under which the matrix products of sufficient length~are rank-one, as it was shown in 
		[6] ({\it Shue,\,Anderson,\,Dey 1998}). We establish a bound on the transient after which any product of matrices whose length exceeds that bound becomes rank-one.
	\end{abstract}
	
	

\section{Introduction}

By max-plus algebra we mean the~linear algebra developed over the~max-plus semiring $\Rmax$, which is the~set $\Rmax=\mathbb{R}\cup \{-\infty\}$ equipped  the~additive operator $a \oplus b = \max\{a,b\}$ and the~multiplicative operator $a \otimes b = a + b$.
We will be mostly interested in the~max-plus matrix multiplication $A\otimes B$ defined for any two matrices $A=(a_{i,j})$ and $B=(b_{i,j})$ with entries in $\Rmax$ of appropriate sizes by the rule
\begin{align*}
(A \otimes B)_{i,j} = \bigoplus_{1 \leq k \leq n} a_{i,k} \otimes b_{k,j} \ = \max_{1 \leq k \leq n} a_{i,k} + b_{k,j}.
\end{align*}
In particular, the $k$th max-plus power of a~square matrix $A$ is defined as
\begin{align*}
	A^{\otimes k} = \underbrace{A \otimes A \otimes \ldots \otimes A}_{(\text{$k$ times})}.
\end{align*}

A lot of work has been done on max-plus powers of a~single matrix. 
Main results of the present paper are in some relation to the~bounds on the ultimate periodicity of the~sequence of max-plus matrix powers $\{A^t\}_{t\geq 1}$, like those established in \cite{sergeev13}, \cite{sergeev14}. However, instead of max-plus powers of a~single matrix we will consider max-plus inhomogeneous matrix products of the~form $A_1\otimes A_2\ldots \otimes \ldots \otimes A_k$ where matrices $A_1,\ldots, A_k$ are taken from an~infinite matrix set $\mathcal{X}$. We will make use of the~assumptions made in~\cite{shue98} and derive a~bound for the {\em rank-one transient} of inhomogeneous products of matrices from $\mathcal{X}$ which is the~minimal $K$ such that $A_1 \otimes A_2 \otimes \ldots \otimes A_k$ for any $k \geq K$ can be represented as a~max-plus outer product $\vec{x}\otimes \vec{y}^\top$, where column vectors $\vec{x}$ and $\vec{y}$ depend on the matrix product. In Theorem \ref{main theorem}
we first obtain a sufficient condition for an inhomogenous product to be rank-one. The bound on the rank-one
transient is then obtained in Corollaries \ref{c:bound} and \ref{c:bound2}.

A practical motivation of this study comes from the switching max-plus dynamical systems of the form
$x(k+1)=A(k)\otimes x(k)$ where matrices $A(k)$ can vary. Such systems arise in some scheduling applications being
related to the way that max-plus algebra is used in modeling discrete event dynamical systems~\cite{Bac}. Let
us also note, in particular, a~recent application of switching max-plus systems of above form in the legged
locomotion of robots~\cite{Kers}, where changing matrices $A(k)$ model the switch of gaits.

\if{
 When looking at a semigroup of matrices then it is harder to determine properties of long matrix products. {\bf(todo: add references - memory-loss property, Merlet's works.)}
 However there has been some work in this area and it a piece of it will be the subject of discussion in the paper.
}\fi

This paper is based on the ideas of~\cite{shue98} where the steady state properties of max-plus inhomogeneous matrix products were considered. The~aim of~\cite{shue98} was to prove that, under certain assumptions, a~sufficiently long max-plus matrix product is rank-one and it can be written as the~outer product of two vectors. Components of these vectors are optimal weights of walks going to and from node 1 respectively.
However, it seems to us that there is an oversight in \cite[Corollary 3.1]{shue98}. This oversight is that in order to prove that the~initial and final parts of an~optimal walk are bounded in length, paper \cite{shue98} uses a~method in which one removes part of a~walk in order to create a~more optimal walk. This would be fine if the matrices were the same however since they are different then removing matrices from the~product changes the~product and one ends up working with a~different product. The result of~\cite{shue98} is also proved for a~sufficient $k$ that is large enough but no concrete bounds are established, so this invited us to look for a~bound on the length of a~max-plus inhomogeneous matrix product after which it becomes an~outer product of two vectors.
Such bound is the~main result of this paper.

The structure of this paper will be as follows. Chapter 2 defines the key ideas and notation that will be used throughout the~paper. In Chapter 3 we introduce and prove the~lemmas required to prove the~main theorem. Chapter 4 contains the~proof of the~main theorem as well as corollaries that follow from the~theorem one being a coarser bound on $k$. Finally, Chapter 5 presents an~example which demonstrates a~long enough inhomogeneous matrix product which is an~outer max-plus product of two vectors.


\section{Definitions and assumptions}

\subsection{Walks and digraphs}

The aim of this subsection is to introduce some important definitions concerning
1) directed weighted graphs, associated with a matrix and 2) trellis digraphs associated with
inhomogeneous matrix products. Note that Definitions~\ref{def:digraphs2} and~\ref{def:walks} are standard~\cite{butkovic08}, and Definition~\ref{def:geom-equiv} follows~\cite{shue98}.

\begin{definition}
\label{def:digraphs}
	{\rm
	A \emph{directed graph} \emph{(digraph)} is a~pair $(\nodes,\edges)$ where $\nodes$ is a~finite set of nodes and $\edges \subseteq \nodes \cross \nodes = \{ (i,j)\colon i,j \in \nodes \}$ is the~set of edges where $(i,j)$ is a~directed edge from node $i$ to node $j$. 
	
	A \emph{weighted digraph} is a~digraph with associated weights $w_{i,j} \in \Rmax$ for each edge $(i,j)$ in the~digraph.
	}
\end{definition}
\begin{definition}
\label{def:digraphs2}	
	{\rm
	A \emph{digraph associated with a~square matrix $A$} is a~digraph $\digr_A = (\nodes_A, \edges_A)$ where the~set $\nodes_A$ has the same number of elements as the number of rows or columns in the~matrix $A$. The~set $\edges_A \subseteq \nodes_A \cross \nodes_A$ is the~set of edges in $\digr_A$ where the~weight of each edge $(i,j)$ is associated with the~respective entry in the~matrix $A$, i.\,e. $w_{i,j} = a_{i,j} \in \Rmax$. If an~entry in the matrix is negative infinity, this means that there is no edge connecting those nodes in that direction.
	}
\end{definition}
\begin{definition}
\label{def:geom-equiv}
	{\rm
	Matrices $A,B\in\Rmax^{n\times n}$ are called \emph{geometrically equivalent} if $E_A=E_B$.
	}
\end{definition}
\begin{definition}
\label{def:walks}
	{\rm
	A sequence of nodes $W=(i_0,\ldots,i_l)$ is called a~\emph{walk on a~weighted digraph} $D=(\nodes,\edges)$ if
	$(i_{s-1},i_s)\in \edges$ for each $s\colon 1\leq s \leq l$.
	This walk is a~\emph{cycle} if the start node $i_0$ and the~end node $i_l$ are the~same.
	It is a~\emph{path} if no two nodes in $i_0,\ldots, i_l$ are the~same. The~\emph{length} of $W$ is $l(W)=l$.  The~\emph{weight} of $W$ is defined as the~max-plus product (i.\,e., the~usual arithmetic sum) of the~weights of each edge $(i_{s-1},i_s)$ traversed throughout the walk, and it is denoted by $p_D(W)$. Note that a~sequence $W=(i_0)$ is also a~walk (without edges), and we assume that it has weight and length $0$.
	
A digraph is \emph{strongly connected} if for any two nodes $i$ and $j$ there exists a walk connecting $i$ to $j$. A matrix is \emph{irreducible} if the graph associated with it in the sense of Definition~\ref{def:digraphs2} is strongly connected.
	}
\end{definition}
\begin{definition}
	{\rm
	The \emph{trellis digraph} $\trellis_{\Gamma(k)} = (\trellisnodes, \trellisedges)$ associated with the~product $\Gamma(k)=A_1\otimes A_2\otimes\ldots\otimes A_k$ is the~digraph with the~set of nodes $\trellisnodes$ and the~set of edges $\trellisedges$, where:
	\begin{itemize}
		\item[(1)] $\trellisnodes$ consists of $k+1$ copies of
		$\nodes$ which are denoted $\nodes_0,\ldots, \nodes_k$, and the~nodes in $\nodes_l$ for each $0\leq l\leq k$ are denoted by $1:l,\ldots, n:l$;
		\item[(2)] $\trellisedges$ is defined by the following rules:
		\begin{itemize}
			\item[a)] there are edges only between $\nodes_l$ and $\nodes_{l+1}$ for each $l$,
			\item[b)] we have $(i:(l-1),j:l)\in \trellisedges$ if and only if $(i,j)$ is an edge of $\digr_{A_l}$, and the~weight of that edge is $(A_l)_{i,j}$.
		\end{itemize}
	\end{itemize}
	The weight of a walk $W$ on $\trellis_{\Gamma(k)}$ is denoted by $p_{\trellis}(W)$.
	}
\end{definition}
\begin{definition}
	{\rm
	Consider a~trellis digraph $\trellis_{\Gamma(k)}$.
	
	By an~\emph{initial walk} connecting $i$ to $j$ on  $\trellis_{\Gamma(k)}$ we mean a~walk on $\trellis_{\Gamma(k)}$ connecting node $i:0$ to $j:m$ where $m$ is the first and the last time the walk arrives at node $j$ and is such that $0\leq m\leq k$.
	
	By a~\emph{final walk} connecting $i$ to $j$ on  $\trellis_{\Gamma(k)}$ we mean a~walk on $\trellis_{\Gamma(k)}$  connecting node $i:l$ to $j:k$, where $l$ is the first and the last time the walk leaves node $i$ and is such that $0\leq l\leq k$.
	
	A~\emph{full walk} connecting $i$ to $j$ on $\trellis_{\Gamma(k)}$ is a~walk on $\trellis_{\Gamma(k)}$ connecting node $i:0$ to $j:k$.
	}
\end{definition}

\subsection{Key notations}
Here we will introduce the~notation that will be used throughout the~paper. We begin by introducing
the following two matrices.

\begin{notation}
	{\rm The ``boundaries'' of $\mathcal{X}$:
	\begin{itemize} \labelsep0.05cm \itemindent0.0cm
	    \item[$A^{\sup}$:]  the~entrywise \emph{supremum over all matrices} in $\mathcal{X}$. More precisely,
	    $A^{\sup}_{ij}=\sup\limits_{X\in\mathcal{X}} (X)_{ij}.$ In max-plus matrix notation,
	    \[
	    A^{\sup}=\bigoplus_{X\in \mathcal{X}} X.
	    \]
	    The weight of a walk $W$ on $\digr_{A^{\sup}}$ will be denoted by $p_{\sup}(W)$.
	    \item[$A^{\inf}$:]  the~entrywise \emph{infimum over all matrices} in $\mathcal{X}$. More precisely,
	    $A^{\inf}_{ij}=\inf\limits_{X\in\mathcal{X}} (X)_{ij}.$
	\end{itemize}
	}
\end{notation}

We now introduce a number of useful parameters. The first group of parameters relates to $\digr_{A^{\sup}}$ and
$\digr_{A^{\inf}}$, and the second to $\trellis_{\Gamma(k)}$.

\begin{notation}
	{\rm
	$\lambda^{\ast}$: the largest cycle mean in the submatrix  $(A^{\sup}_{i,j})_{i,j \neq 1}$:
	\begin{equation*}
	\lambda^*=\max\limits_{k\geq 1}\left( \max\limits_{2\leq i_1,\ldots,i_k\leq n} \frac{A^{\sup}_{i_1i_2}+\ldots+ A^{\sup}_{i_ki_1}}{k} \right).
	\end{equation*}
	}
\end{notation}

\begin{notation}
	{\rm
	Weights of some paths and walks on $\digr_{A^{\sup}}$ and $\digr_{A^{\inf}}$:
	\begin{itemize}
	    \item[$\alpha_{i}$:] the maximal weight of paths on $\digr_{A^{\sup}}$  connecting $i$ to $1$;
	    \item[$\beta_{j}$:] the maximal weight of paths on $\digr_{A^{\sup}}$ connecting $1$ to $j$;
	    \item[$\gamma_{ij}$:] the maximal weight of paths on $\digr_{A^{\sup}}$ connecting $i$ to $j$ and not going through node $1$;
	    \item[$w_{i}$:] the maximal weight of walks of length not exceeding $k$ on
			$\digr_{A^{\inf}}$ connecting $i$ to $1$;
	    \item[$v_{j}$:] the maximal weight of walks of length not exceeding $k$ on
			$\digr_{A^{\inf}}$ connecting $1$ to $j$.
	\end{itemize}
	}
\end{notation}

\begin{notation}
	{\rm
	Weight of optimal walks on $\trellis_{\Gamma(k)}$:
	\begin{itemize} \labelsep0.05cm \itemindent0.2cm
		\item[$w_{i}^{\ast}\colon$] the maximal weight of initial walks on $\trellis_{\Gamma(k)}$ connecting $i$ to $1$;
		\item[$v_{j}^{\ast}\colon$] the maximal weight of final walks on $\trellis_{\Gamma(k)}$ connecting $1$ to $j$.
	\end{itemize}
	}
\end{notation}
Note that the length of any walk on $\trellis_{\Gamma(k)}$ does not exceed $k$.

\subsection{Key assumptions}
We will use the following main assumptions, which are very similar to those of~\cite{shue98}.
\begin{assumption} \label{a:geom}
	{\rm
	The matrices $A_{i}, \: i\in {1,...,k}$ are chosen from a~set $\mathcal{X}$
	of~geometrically equivalent irreducible matrices, and the matrix $A^{\inf}$ is also geometrically equivalent to any of them.
	}
\end{assumption}

\begin{assumption} \label{a:oneloop}
	{\rm
	The digraph of each matrix in the set $\mathcal{X}$ has a unique critical cycle of length $1$ at node $1$
	with weight $0$.
	}
\end{assumption}

\begin{assumption} \label{a:suploop}
	{\rm
	The digraph of the matrix $A^{\sup}$
	has a unique critical cycle of~length $1$ at node $1$ of weight $0$.
	}
\end{assumption}

Note that, if the~unique critical cycle of length one is at any other node than $1$, then none of our main results
will change significantly.
However, the~given assumptions are still very limiting in terms of the~type of matrix and the~real world situation that this can apply to.

\section{Preliminary lemmas}

The aim of this section is to prove some preliminary lemmas which will help us to
construct the terms in the bound of Theorem~\ref{main theorem} and its corollaries. The main ideas are that
the lengths of optimal initial and final walks are bounded (Lemmas~\ref{l:bound-initial} and~\ref{l:bound-final})
and that after some transient on the length any optimal full walk should pass through node $1$ (Lemma~\ref{l:bound-full}).

\begin{lemma} \label{l:bound-initial}
	Let $W_1$ be an optimal initial walk on trellis digraph $\trellis_{\Gamma(k)}$ connecting $i$ to $1$. Then we have the~following upper bound on its length:
	\begin{equation}
		l(W_1)\leq \frac{w_i^*-\alpha_i}{\lambda^{\ast}}+(n-1).
	\end{equation}
\end{lemma}
\begin{Proof}
	 Due to Assumptions~\ref{a:oneloop} and \ref{a:suploop} we have $\lambda^{\ast}<0$.  The~weight of any optimal walk $W_1$ connecting $i$ to $1$ is less than or equal to that of a~path $P_1$ connecting $i$ to $1$ on $\digr_{A^{\sup}}$ plus the~remaining length multiplied by $\lambda^{\ast}<0$. Thus
	\begin{equation*}
	  p_{\trellis}(W_1) \leq p_{\sup}(P_1) + (l(W_1)-(n-1))\lambda^{\ast}.
	\end{equation*}
	Next we bound $p_{\sup}(P_1)\leq\alpha_i$, hence
	\begin{equation}
	\label{e:bound1}
	  p_{\trellis}(W_1)\leq \alpha_i + (l(W_1)-(n-1))\lambda^{\ast}.
	\end{equation}
	Now assume by contradiction that $l(W_1) > \frac{w_i^*-\alpha_i}{\lambda^{\ast}}+(n-1)$.
	However, this is equivalent to
	\begin{equation}
	\label{e:bound11}
	\alpha_i + (l(W_1)-(n-1))\lambda^{\ast}< w_i^*.
\end{equation}
Combining~\eqref{e:bound1} with~\eqref{e:bound11} we obtain
$p_{\trellis}(W_1)<w_i^*$ meaning that $W_1$ is not optimal, a~contradiction.
The proof is complete.
\end{Proof}

\bigskip

We now state an analogous lemma on the~length of~an~optimal final walk. The proof is similar and will be omitted.

\begin{lemma} \label{l:bound-final}
	Let $W_2$ be an~optimal final walk on trellis digraph $\trellis_{\Gamma(k)}$ connecting $1$ to $j$. Then we have the following upper bound on its length:
	\begin{equation}
		l(W_2)\leq \frac{v_j^*-\beta_j}{\lambda^{\ast}}+(n-1).
	\end{equation}
\end{lemma}

\begin{lemma} \label{l:bound-full}
	Let
	\begin{equation}
	    k > \frac{w_{i}^{\ast}-\alpha_{i}+v_{j}^{\ast}-\beta_{j}}{\lambda^{\ast}} +2(n-1).
	\end{equation}
	Then any optimal full walk $W$ connecting $i$ to $j$ on $\trellis_{\Gamma(k)}$ and going through node $1$ is decomposed as,
	\begin{align*}
	    W = W_{1} \circ C \circ W_{2}
	\end{align*}
	where $W_1$ is an~optimal initial walk and $W_2$ is an~optimal final walk
	which satisfy
	\begin{align*}
	    l(W_{1}) & \leq  \frac{w_{i}^{\ast}-\alpha_{i}}{\lambda^{\ast}} +(n-1), \\
	    l(W_{2}) & \leq  \frac{v_{j}^{\ast}-\beta_{j}}{\lambda^{\ast}} +(n-1),
	\end{align*}
	$C$ consists of several loops $1 \to 1$  and
	\begin{align*}
	    p_{\trellis}(W) = w_{i}^{\ast} +v_{j}^{\ast}.
	\end{align*}
\end{lemma}
\begin{Proof}
	Let $W$ be an~optimal full walk connecting $i$ to $j$ that traverses node $1$ at least once.
	Note first that all edges between the~first and the~last occurrence of $1$ in $W$ can be replaced with the~copies of $(1,1)$, since these edges are present in every matrix $X_{\alpha}$ from $\mathcal{X}.$ 
	Assumption~\ref{a:suploop} implies that this leads to a~strict increase of the~weight, therefore we must have $W=\Tilde{W}_{1} \circ \Tilde{C} \circ \Tilde{W}_{2},$ where $\Tilde{C}$ consists of~several edges $(1,1)$, $\Tilde{W}_1$ is an~initial walk from $i$ to $1$ and $\Tilde{W}_2$ is a~final walk from $1$ to $j$. We have $p_{\trellis}(\Tilde{C})=0$, so $p_{\trellis}(W)=p_{\trellis}(\Tilde{W}_1)+p_{\trellis}(\Tilde{W}_2)$.
	
	Now we note that by Lemmas~\ref{l:bound-initial} and~\ref{l:bound-final} the~length $k$ is sufficient for constructing a~walk $W'=V_{1} \circ C' \circ V_{2}$ where $V_1$ is an~optimal initial walk from $i$ to $1$, $C'$ consists of~several copies of $(1,1)$ and
	$V_2$ is an~optimal final walk from $1$ to $j$. The weight of~this walk is $w_i^*+v_j^*$.
	
	By the optimality of $V_1$ and $V_2$ we have $p_{\trellis}(\Tilde{W}_1)\leq p_{\trellis}(V_1)$ and $p_{\trellis}(\Tilde{W}_2)\leq p_{\trellis}(V_2)$. Since $W$ is optimal, both inequalities should hold with equality.
	
	That is, $\Tilde{W}_1$ is an~optimal initial walk connecting $i$ to $1$ and $\Tilde{W}_2$ is an~optimal final walk connecting $1$ to $j$, so that $\Tilde{W}_1$, $\Tilde{W}_2$
	and $\Tilde{C}$ can be taken for $W_1$, $W_2$ and $C$ respectively.
	The proof is complete.
\end{Proof}

\begin{lemma} \label{i:ineq}
	Let
	\begin{align} \label{e:ineq}
	    k > \frac{w_{i}^{\ast} +v_{j}^{\ast} -\gamma_{i,j}}{\lambda^{\ast}} +(n-1) .
	\end{align}
	Then any full walk $W$ connecting $i$ to $j$ on $\trellis_{\Gamma(k)}$ that does not go through node $1$ has weight smaller than $w_i^{\ast}+v_j^{\ast}$.
\end{lemma}
\begin{Proof}
	 Due to Assumption 3.1 and 3.2 the~weight of any walk $p_{\trellis}(W)$ connecting $i \to j$ and not going through $1$ will be less than or equal to the~weight of a~path $P$ on $\digr_{A^{\sup}}$ going from $i$ to $j$ plus the~remaining length multiplied by $\lambda^{\ast}$:
	\begin{align} \label{e:bound-first}
	    p_{\trellis}(W) \leq p_{\sup}(P) + (k-(n-1))\lambda^{\ast}.
	\end{align}
	As $P$ is a path from $i \to j$, its weight is bounded above by $\gamma_{ij}$. Therefore \begin{align} \label{e:bound3}
	    p_{\sup}(P) + (k-(n-1))\lambda^{\ast} \leq \gamma_{ij} + (k-(n-1))\lambda^{\ast}.
	\end{align}
	We now see that~\eqref{e:ineq} is equivalent to
	\begin{align} \label{e:bound4}
		\gamma_{ij} + (k-(n-1))\lambda^{\ast} < w_i^*+v_j^*.
	\end{align}
	Combining~\eqref{e:bound-first},\eqref{e:bound3} and~\eqref{e:bound4} we see that $p_{\trellis}(W)<w_i^{\ast}+v_j^{\ast}$, thus the~proof is complete.\\ \makebox{}
\end{Proof}
\enlargethispage{5mm}

\section{Main results}
Now we can move on to the~main theorem of the~paper, with its modifications and corollaries.
\begin{theorem} \label{main theorem}
	Let $\Gamma(k)$ be an~inhomogenous max-plus matrix product $\Gamma(k) = A_1 \otimes A_2 \otimes \ldots \otimes A_k$ with $k$ satisfying
	\begin{align} \label{main condition}
		k > \max \Big( \frac{w_{i}^{\ast} + v_{j}^{\ast} - \gamma_{ij}}{\lambda^{\ast}} + (n-1), \frac{w_{i}^{\ast} - \alpha_{i} + v_{j}^{\ast} - \beta_{j}}{\lambda^{\ast}} + 2(n-1)) \Big)
	\end{align}
	for some $i,j\in\nodes$, then
	\begin{align*}
	    \Gamma(k)_{i,j} & = \Gamma(k)_{i,1} \otimes \Gamma(k)_{1,j} \\
	    & = \Gamma(k)_{i,1} + \Gamma(k)_{1,j}.
	\end{align*}
\end{theorem}
\begin{Proof}
	As seen by Lemma \ref{i:ineq}, if
	\begin{align*}
		k > \frac{w_{i}^{\ast} + v_{j}^{\ast} - \gamma_{ij}}{\lambda^{\ast}} + (n-1)
	\end{align*}
	then any walk on $\trellis_{\Gamma(k)}$ not going through node $1$ will have weight smaller than $w_i^{\ast}+v_j^{\ast}$. By Lemma \ref{l:bound-full}, if 
	\begin{align*}
		k > \frac{w_{i}^{\ast} - \alpha_{i} + v_{j}^{\ast} - \beta_{j}}{\lambda^{\ast}} + 2(n-1)
	\end{align*}
	then any optimal full walk going through node $1$ will consist of the three parts $W_{1},W_{2}$ and $C$ as defined in the Lemma and its weight will be $w_{i}^{\ast} + v_{j}^{\ast}$.
	Hence if $k$ satisfies both inequalities then any optimal full walk goes through node$1$ and has weight
	\begin{align*}
		\Gamma(k)_{ij} = w_{i}^{\ast} + v_{j}^{\ast}
	\end{align*}
	
	Observe that  $w_{1}^{\ast}$ and $v_{1}^{\ast}$ are equal to $0$, since the~weight of any optimal initial or final walk on $\trellis_{\Gamma(k)}$ connecting $1$ to $1$ is $0$. Therefore
	\begin{align*}
	    \Gamma(k)_{i,1} & = w_{i}^{\ast}+ v_{1}^{\ast} = w_{i}^{\ast}, \\
	    \Gamma(k)_{1,j} & = w_{1}^{\ast}+ v_{j}^{\ast} = v_{j}^{\ast},
	\end{align*}
	 and
	\begin{align*}
		\Gamma(k)_{i,j} = \Gamma(k)_{i,1} + \Gamma(k)_{1,j}. 
	\end{align*}
\end{Proof}

\pagebreak

Let us extend Theorem~\ref{main theorem} to a~matrix form.

\begin{corollary} \label{m:form}
	If the matrix product $\Gamma(k)$ with length $k$ satisfies
	\begin{align*}
		k > \max_{i,j \in \nodes} \Big( \frac{w_{i}^{\ast} + v_{j}^{\ast} - \gamma_{ij}}{\lambda^{\ast}} + (n-1), \frac{w_{i}^{\ast} - \alpha_{i} + v_{j}^{\ast} - \beta_{j}}{\lambda^{\ast}} + 2(n-1)) \Big)
	\end{align*}
	for all $i,j \in \nodes$, then $\Gamma(k)$ is rank one and
	\[
    \Gamma(k) =
    \begin{bmatrix}
       \Gamma(k)_{1,1} \\
       \Gamma(k)_{2,1} \\
       \vdots \\
       \Gamma(k)_{n,1}
    \end{bmatrix} \otimes
    \begin{bmatrix}
    	\Gamma(k)_{1,1} & \Gamma(k)_{1,2} & \ldots & \Gamma(k)_{1,n}
	\end{bmatrix}.
	\]
\end{corollary}
\begin{Proof}
	Using Theorem \ref{main theorem} for all $i,j \in \nodes$, if $k$ satisfies the~condition~\eqref{main condition} then
	\begin{align*}
	    \Gamma(k)_{i,j} = \Gamma(k)_{i,1} + \Gamma(k)_{1,j}.
	\end{align*}
	Since this applies for all $i,j \in \nodes$, $\Gamma(k)_{i,1}$ and $\Gamma(k)_{1,j}$ can be written as vectors in $\mathbb{R}^n$. Using the~max-plus outer product of~these two vectors it becomes
	\[
	\Gamma(k) =
	\begin{bmatrix}
		\Gamma(k)_{1,1} \\
	    \Gamma(k)_{2,1} \\
	    \vdots \\
	    \Gamma(k)_{n,1}
	\end{bmatrix} \otimes
	\begin{bmatrix}
		\Gamma(k)_{1,1} \\ \Gamma(k)_{1,2} \\ \vdots \\ \Gamma(k)_{1,n}
	\end{bmatrix}^{\top}
	\]
	thus proving the corollary.
\end{Proof}

\bigskip

The bounds of Theorem~\ref{main theorem} and Corollary~\ref{m:form} are interesting to see but they are implicit. This is because in order to calculate $w_{i}^{\ast}$ and $v_{j}^{\ast}$ you need to calculate $\Gamma(k)$ in which the~length of~the~product is dictated by the~bound using $w_{i}^{\ast}$ and $v_{j}^{\ast}$. However another bound can be derived from Theorem \ref{main theorem} using $A^{\inf}$. From the~definition of~$A^{\inf}$, $w_{i}$ and $v_{j}$ it is easy to see that for all $i,j \in \nodes$
\begin{align*}
    w_{i} \leq w_{i}^{\ast} \; \text{ and } \; v_{j} \leq v_{j}^{\ast}
\end{align*}
These inequalities, together with Theorem~\ref{main theorem}, imply the following results.

\begin{corollary} \label{c:bound}
	Let $\Gamma(k)$ be an~inhomogenous max-plus matrix product $\Gamma(k) = A_1 \otimes A_2 \otimes \ldots \otimes A_k$ with $k$ satisfying
	\begin{align*}
	k > \max \Big( \frac{w_{i} + v_{j} - \gamma_{ij}}{\lambda^{\ast}} + (n-1), \frac{w_{i} - \alpha_{i} + v_{j} - \beta_{j}}{\lambda^{\ast}} + 2(n-1)) \Big)
	\end{align*}
	for some $i,j\in\nodes$, then
	\begin{align*}
	    \Gamma(k)_{i,j} & = \Gamma(k)_{i,1} \otimes \Gamma(k)_{1,j} \\
	    & = \Gamma(k)_{i,1} + \Gamma(k)_{1,j}.
	\end{align*}
\end{corollary}
We now also extend Corollary \ref{c:bound} to a~matrix form.
\begin{corollary} \label{c:bound2}
	Let $\Gamma(k)$ be an~inhomogenous max-plus matrix product $\Gamma(k) = A_1 \otimes A_2 \otimes \ldots \otimes A_k$ with $k$ satisfying
	\begin{align*}
	k > \max_{i,j \in \nodes} \Big( \frac{w_{i} + v_{j} - \gamma_{ij}}{\lambda^{\ast}} + (n-1), \frac{w_{i} - \alpha_{i} + v_{j} - \beta_{j}}{\lambda^{\ast}} + 2(n-1)) \Big)
	\end{align*}
	then $\Gamma(k)$ is rank one and
	\[
	\Gamma(k) =
	\begin{bmatrix}
		\Gamma(k)_{1,1} \\
	    \Gamma(k)_{2,1} \\
	    \vdots \\
	    \Gamma(k)_{n,1}
	\end{bmatrix} \otimes
	\begin{bmatrix}
		\Gamma(k)_{1,1} & \Gamma(k)_{1,2} & \ldots & \Gamma(k)_{1,n}
	\end{bmatrix}.
	\]
\end{corollary}
\smallskip

Note that this bound is explicit, and in particular it can be found numerically without having to calculate $\Gamma(k)$ beforehand. This is a bound for the rank-one transient of inhomogeneous products.

\section{An example}
To illustrate what has been achieved in the paper let us consider an~example. Let $D_A$ be a~digraph consisting of five nodes with the~generalised associated weight matrix (for convention let $\varepsilon = -\infty$),
\[
A =
\begin{bmatrix}
	a_{1,1} & a_{1,2} & a_{1,3} & \varepsilon & \varepsilon \\
    a_{2,1} & \varepsilon & \varepsilon & \varepsilon & a_{2,5} \\
    \varepsilon & \varepsilon & \varepsilon & a_{3,4} & \varepsilon \\
    \varepsilon & a_{4,2} & \varepsilon & \varepsilon & \varepsilon \\
    a_{5,1} &\varepsilon & \varepsilon & a_{5,4} & \varepsilon
\end{bmatrix},
\]
where $a_{i,j} \in \Rmax$. Consider the~set $\mathcal{X} = \{A_{1}, A_{2}, A_{3}\}$ where

\[
A_{1} =
\begin{bmatrix}
	0 & -1 & -2 & \varepsilon & \varepsilon \\
    -3 & \varepsilon & \varepsilon & \varepsilon & -3 \\
    \varepsilon & \varepsilon & \varepsilon & -4 & \varepsilon \\
    \varepsilon & -5 & \varepsilon & \varepsilon & \varepsilon \\
    -6 &\varepsilon & \varepsilon & -5 & \varepsilon
\end{bmatrix},
\]
\[
A_{2} =
\begin{bmatrix}
	0 & -4 & -3 & \varepsilon & \varepsilon \\
    -4 & \varepsilon & \varepsilon & \varepsilon & -3 \\
    \varepsilon & \varepsilon & \varepsilon & -2 & \varepsilon \\
    \varepsilon & -1 & \varepsilon & \varepsilon & \varepsilon \\
    -1 &\varepsilon & \varepsilon & 1 & \varepsilon
\end{bmatrix},
\]
\[
A_{3} =
\begin{bmatrix}
	0 & 2 & -4 & \varepsilon & \varepsilon \\
    -5 & \varepsilon & \varepsilon & \varepsilon & -6 \\
    \varepsilon & \varepsilon & \varepsilon & -4 & \varepsilon \\
    \varepsilon & -3 & \varepsilon & \varepsilon & \varepsilon \\
    -2 &\varepsilon & \varepsilon & 2 & \varepsilon
\end{bmatrix}.
\]
It can be seen that these satisfy the~assumptions with the top left entry of each matrix being zero. Using these we can calculate the~coarser bounds of Corollaries~\ref{c:bound} and~\ref{c:bound2}. In order to do that we need $A^{\sup}$ and $A^{\inf}$, which are
\[
A^{\sup} =
\begin{bmatrix}
	0 & 2 & -2 & \varepsilon & \varepsilon \\
    -3 & \varepsilon & \varepsilon & \varepsilon & -3 \\
    \varepsilon & \varepsilon & \varepsilon & -2 & \varepsilon \\
    \varepsilon & -1 & \varepsilon & \varepsilon & \varepsilon \\
    -1 &\varepsilon & \varepsilon & 2 & \varepsilon
\end{bmatrix} \quad \text{and} \quad
A^{\inf} =
\begin{bmatrix}
	0 & -4 & -4 & \varepsilon & \varepsilon \\
    -5 & \varepsilon & \varepsilon & \varepsilon & -6 \\
    \varepsilon & \varepsilon & \varepsilon & -4 & \varepsilon \\
    \varepsilon & -5 & \varepsilon & \varepsilon & \varepsilon \\
    -6 &\varepsilon & \varepsilon & -5 & \varepsilon
   \end{bmatrix}.
\]
We now begin to calculate the bounds of Corollaries~\ref{c:bound} and~\ref{c:bound2}. The only cycle that does not got through node $1$ is $(2 \to 5 \to 4 \to 2)$ which has average weight $\lambda^{\ast}=-\frac{2}{3}$. Using $A^{\sup}$ we get $\alpha_{i}$, $\beta_{j}$ and $\gamma_{i,j}$ as the entries of
\[
\alpha =
\begin{bmatrix}
	0 \\
    -3 \\
    -6 \\
    -4 \\
    -1
\end{bmatrix}, \;
\beta =
\begin{bmatrix}
	0 \\
    2 \\
    -2 \\
    1 \\
    -1
\end{bmatrix}, \;
\gamma =
\begin{bmatrix}
	\varepsilon& \varepsilon & \varepsilon & \varepsilon & \varepsilon \\
    \varepsilon & -2 & \varepsilon & -1 & -3 \\
    \varepsilon & -3 & \varepsilon & -2 & -6 \\
    \varepsilon & -1 & \varepsilon & -2 & -4 \\
    \varepsilon & 1 & \varepsilon & 2 & -2
\end{bmatrix}.
\]
Using $A^{\inf}$ we can also calculate $w_{i}$ and $v_{j}$ as the
entries of
\[
w =
\begin{bmatrix}
	0 \\
    -4 \\
    -13 \\
    -9 \\
    -6
\end{bmatrix}, \;
v =
\begin{bmatrix}
	0 \\
    -5 \\
    -4 \\
    -8 \\
    -10
\end{bmatrix}.
\]
With these pieces we can construct the bounds for $k$ for each combination of $i$ and $j$:
\begin{align*}
	k > \max_{i,j\in \nodes} \left(
    \begin{bmatrix}
	    \varepsilon & \varepsilon & \varepsilon & \varepsilon & \varepsilon \\
	    \varepsilon & 14.5 & \varepsilon & 20.5 & 20.5 \\
	    \varepsilon & 26.5 & \varepsilon & 32.5 & 29.5 \\
	    \varepsilon & 23.5 & \varepsilon & 26.5 & 26.5 \\
	    \varepsilon & 22 & \varepsilon & 28 & 25
    \end{bmatrix},
    \begin{bmatrix}
	    8 & 18.5 & 11 & 21.5 & 21 \\
	    9 & 20 & 12.5 & 23 & 23 \\
	    18 & 29 & 21.5 & 32 & 32 \\
	    15 & 26 & 18.5 & 29 & 29 \\
	    15.5 & 26 & 18.5 & 29 & 29
    \end{bmatrix}
    \right) \Leftrightarrow k > 32.
\end{align*}
This means that if a~matrix product $\Gamma(k)$ has length greater then $32$ then it will be rank-one.
Let us now take a random product of length $44$:
\begin{align*}
    \Gamma(k) & = A_{1} \otimes A_{3} \otimes A_{1} \otimes A_{2} \otimes A_{3} \otimes A_{1} \otimes A_{2} \otimes A_{2} \otimes A_{1} \otimes A_{3} \otimes A_{1} \\
    & \otimes A_{2} \otimes A_{1} \otimes A_{2} \otimes A_{3} \otimes A_{3} \otimes A_{1} \otimes A_{2} \otimes A_{1} \otimes A_{1} \otimes A_{3} \otimes A_{2} \\
    & \otimes A_{3} \otimes A_{2} \otimes A_{2} \otimes A_{3} \otimes A_{1} \otimes A_{1} \otimes A_{2} \otimes A_{3} \otimes A_{2} \otimes A_{1} \otimes A_{3} \\
    & \otimes A_{1} \otimes A_{2} \otimes A_{3} \otimes A_{1} \otimes A_{3} \otimes A_{3} \otimes A_{1} \otimes A_{2} \otimes A_{2} \otimes A_{1} \otimes A_{1}.
\end{align*}
We obtain that
\[
\Gamma(k) =
\begin{bmatrix}
	0 & -1 & -2 & -6 & -4 \\
    -3 & -4 & -5 & -9 & -7 \\
    -10 & -11 & -12 & -16 & -14 \\
    -10 & -11 & -12 & -16 & -14 \\
    -6 & -7 & -8 & -12 & -10
\end{bmatrix}.
\]
We see that $\Gamma(k)=w_{i}^{\ast} \otimes (v_{j}^{\ast})^{\top} = \Gamma(k)_{i,1} \otimes (\Gamma(k)_{1,j})^{\top}$ where
\[
w^{\ast} =
\begin{bmatrix}
	0 \\
    -3 \\
    -10 \\
    -10 \\
    -6
\end{bmatrix}, \;
v^{\ast} =
\begin{bmatrix}
	0 \\
    -1 \\
    -2 \\
    -6 \\
    -4
\end{bmatrix}.
\]
Note that the bound appearing in Corollary~\ref{m:form} is equal to
\begin{align*}
	\max_{i,j\in \nodes} \left(
	\begin{bmatrix}
     \varepsilon & \varepsilon & \varepsilon & \varepsilon & \varepsilon \\
     \varepsilon & 7 & \varepsilon & 16 & 10 \\
     \varepsilon & 16 & \varepsilon & 25 & 16 \\
     \varepsilon & 19 & \varepsilon & 25 & 19 \\
     \varepsilon & 16 & \varepsilon & 25 & 16
    \end{bmatrix},
    \begin{bmatrix}
	    8 & 12.5 & 8 & 18.5 & 12.5 \\
	    8 & 12.5 & 8 & 18.5 & 12.5 \\
	    14 & 18.5 & 14 & 24.5 & 18.5 \\
	    17 & 21.5 & 17 & 27.5 & 21.5 \\
	    15.5 & 20 & 15.5 & 26 & 20
    \end{bmatrix}
    \right) = 27.5,
\end{align*}
which is smaller than the~coarser bound $33$.
\medskip
\section*{Acknowledgement}
\small
This work was supported by EPSRC Grant EP/P019676/1. We would like to thank Dr. Oliver Mason for giving us an~idea for this paper. We are also grateful to Dr. Glenn Merlet and our anonymous referees for their careful reading, useful suggestions and advice.

\makesubmdate

\makecontacts


\begin{thebibliography}{000}

\bibitem{Bac}
F.\,L.~Baccelli, G.~Cohen, G.\,J. Olsder, and J.\,P. Quadrat:
\newblock{Synchronization and Linearity: An Algebra for
Discrete Event Systems.}
    \newblock{John Wiley and Sons, Hoboken 1992.}

\bibitem{butkovic08}
    P.~Butkovic:
    \newblock{Max-linear Systems: Theory and Algorithms.}
    \newblock{Springer Monographs in Mathematics, London 2010.}
\href{http://dx.doi.org/10.1007/978-1-84996-299-5}{DOI:10.1007/978-1-84996-299-5}

\bibitem{Kers}
B.~Kersbergen:
\newblock{Modeling and Control of Switching Max-Plus-Linear Systems.}
\newblock{Ph.D. Thesis, TU Delft 2015.}

\bibitem{sergeev13}
    G.~Merlet, T.~Nowak, and S.~Sergeev:
    \newblock{Weak CSR expansions and transience bounds in max-plus algebra.}
    \newblock{Linear Algebra Appl. {\mi 461} (2014), 163--199.}
\href{http://dx.doi.org/10.1016/j.laa.2014.07.027}{DOI:10.1016/j.laa.2014.07.027}

\bibitem{sergeev14}
G.~Merlet, T.~Nowak, H.~Schneider, and S.~Sergeev:
\newblock{Generalizations of bounds on the index of convergence to weighted digraphs.}
\newblock{Discrete Appl. Math. {\mi 178} (2014), 121--134.}
\href{http://dx.doi.org/10.1016/j.dam.2014.06.026}{DOI:10.1016/j.dam.2014.06.026}

\bibitem{shue98}
	L.~Shue, B.\,D.\,O.~Anderson, and S.~Dey:
	\newblock{On steady state properties of certain max-plus products.}
	\newblock{In: Proc. American Control Conference, Philadelphia, Pensylvania 1998, pp.\,1909--1913.}
\href{http://dx.doi.org/10.1109/acc.1998.707354}{DOI:10.1109/acc.1998.707354}

\end{thebibliography}
\end{document}